\documentclass{amsart}
\usepackage{graphicx} 
    \usepackage{amsrefs, amssymb}

\theoremstyle{plain}
 \newtheorem{theorem}{Theorem}
\newtheorem{proposition}[theorem]{Proposition}

\newtheorem{lemma}[theorem]{Lemma}
\newtheorem{definition}[theorem]{Definition}
\theoremstyle{remark}

\def\dom{\mathop{dom}}
\author{Alan Dow}
\address{UNC Charlotte, Dept. of Mathematics and Statistics}
\email{adow@charlotte.edu}

\author[I. Juh\'asz]{Istv\'an Juh\'asz}
\address      {HUN-REN Alfr\'ed Rényi Institute of Mathematics}
\email{juhasz@renyi.hu}

\makeatletter
  \@namedef{subjclassname@2020}{%
  \textup{2020} Mathematics Subject Classification}
   \makeatother


\title[weight and net weight]{Weight, net weight, and elementary submodels}

\begin{document}

\begin{abstract}
In this note we prove several theorems that are
related to some results and problems from \cite{JSSz}.

We answer two of the main problems that were
raised in \cite{JSSz}.
First we give a ZFC example of a {\em Hausdorff} space in $C(\omega_1)$ that has uncountable net weight.
Then we prove that after adding any number of Cohen reals to a model of CH, in the extension every {\em regular}
space in $C(\omega_1)$ has countable net weight.

We prove in ZFC that for any regular topology of
uncountable weight on $\omega_1$
there is a non-stationary subset that has uncountable weight as well.
Moreover, if all final segments of $\omega_1$ have uncountable weight
then the assumption of regularity can be dropped.
By \cite{JSSz}, the analogous statements for the net weight are independent from ZFC.

Our proofs of all these results make essential use of elementary submodels.
\end{abstract}

\maketitle

The authors dedicate this paper to the memory of
Peter Nyikos, our long-time
mathematical friend and colleague, and in recognition of his
many significant contributions to topology.

\bigskip

This article continues the investigation of some interesting
strengthenings
  of the countable chain condition
  property which, on the other hand, follow from having a countable network.
A key concept, introduced
under a different name by Tkachenko \cite{Tkachenko78},
is that of a hereditarily good (HG) topological space $Y$,
which means that whenever a set of $\aleph_1$-many points from $Y$
together with an assignment of a neighborhood to each of
these points is given, there are two points
which are contained in each other's prescribed neighborhoods.
This property is sometimes also called {\em pointed CCC}.
Clearly such spaces are both hereditarily separable and hereditarily
Lindel\"of. The
  Sorgenfrey line is an example of a hereditarily separable and hereditarily
Lindel\"of space that is not HG
  and Tkachenko asked about the connection of HG to having countable net weight.
  Recall that a family $\mathcal N$ of subsets of a space
  $X$ is a network if for each point $x$ and neighborhood
   $U$, there is an $N\in\mathcal N$ with $x\in N\subset U$.
   The net weight $nw(X)$ of a space $X$ is the smallest cardinality
   of a network for $X$.

  The property of HG was further strengthened in
  \cite{HKsup} and re-named and studied in \cite{JSSz} as follows.

\begin{definition}
\begin{enumerate}
\item  A space $X$ is in $C(\omega_1)$
 if for every   partial neighborhood
 assignment $U$ with
  $\mathop{dom}(U)\subset X$ uncountable,
   there is an uncountable
    $Y\subset \mathop{dom}(U)$
    satisfying that
     $y\in U(z)$ for all
      $y,z\in Y$.

\item A space $X$ is in $N(\omega_1)$
 if every size $\omega_1$ subspace of it
 has a countable network.

\end{enumerate}
\end{definition}

Clearly a space with a countable network is in $C(\omega_1)$, consequently we have
$$\{X : nw(X) = \omega\} \subset N(\omega_1) \subset C(\omega_1).$$

In this paper regular spaces are always assumed to be Hausdorff.

Answering a question of Hart and Kunen, raised in
both \cite{HKsup} and \cite{HKsupdup}, it was proved in
\cite{JSSz} Theorem 1.5 that under CH
every regular space in $C(\omega_1)$ has a countable network, i.e. has countable net weight.
(Actually, this was shown to follow from the principle {\em super stick}, a consequence of CH
that is actually consistent with the negation of CH.) Problem (1) on p. 4 of \cite{JSSz} asks if in this
result {\em regular} can be weakened to {\em Hausdorff}. This problem is
solved in Theorem \ref{negative} below.

\begin{theorem} There is a refinement\label{negative} of the topology
on the reals that is in $C(\omega_1)$ and   does
not have countable net weight.
\end{theorem}

\begin{proof}  The topology $\tau$ on $\mathbb{R}$ will be quite simple.
For each $r\in \mathbb R$, we will choose a sequence
 $a_r\subset \mathbb R\setminus \{r\}$
 that Euclidean converges to $r$.  Then a neighborhood base in $\tau$
 at $r\in\mathbb R$,
 will be the family
   $$\{ (q,s) \setminus a_r : (q,s\in \mathbb Q) \text{ and } q<r<s\}.$$

    The family of sequences $\{ a_r : r\in \mathbb R\}$ will be chosen
    to satisfy:

\smallskip

(*) \ For any countable collection $\{T_n : n\in \omega\}  \subset [\mathbb R]^\mathfrak{c}$,
  there is a set $X\subset \mathbb R$ of cardinality $\mathfrak c$
  such that for each $n < \omega$ and
   $r\in X$ we have    $a_r\cap T_n \ne \emptyset$,  provided that
   $q<r<s$ always implies $|(q,s)\cap T_n| = \mathfrak{c}$, i.e.
   $r$ is a complete accumulation point of $T_n$ in the usual Euclidean topology of $\mathbb R$.

   For any $T \in [\mathbb R]^\mathfrak{c}$, we shall denote by $T^\circ$ the set of all
   complete accumulation points of $T$. Note that we have $|T \setminus T^\circ| < \mathfrak{c}$
   for any $T \in [\mathbb R]^\mathfrak{c}$.

   \medskip

   We now check that having (*) will imply that $\tau$
   does not have a countable network.
   Assume that $\mathcal{S}$ is any countable family of subsets  of $\mathbb R$,
   moreover let $\mathcal{T} = \{T \in \mathcal{S} : |T| = \mathfrak{c} \}$.
   If $\mathcal{T} = \emptyset$ then, of course, $\mathcal{S}$ doesn't even cover $\mathbb{R}$,
   and so cannot be a network for any topology on $\mathbb{R}$. So, we assume that $\mathcal{T} \ne \emptyset$.

   Let us put $$Y = \bigcup \{S \in \mathcal{S} : |S| < \mathfrak{c} \} \cup \bigcup \{T \setminus T^\circ : T \in \mathcal{T}\}. $$
   Then we
   clearly have $|Y| < \mathfrak{c}$.
   Now let  $\{T_n : n\in \omega\}$
    enumerate the family $\mathcal{T}$,
    and let $X \in [\mathbb R]^\mathfrak{c}$ be as in (*). We may then pick $r \in X \setminus Y$.

 Clearly then $r \in S \in \mathcal{S}$ implies that $S = T_n$ for some $n < \omega$, moreover
 $r$ is a complete accumulation point of $T_n$, hence $a_r\cap T_n \ne \emptyset$.
 It follows then that $\mathbb R\setminus
  a_r$ is a $\tau$-neighborhood of $r$ that fails to contain
  any $S \in \mathcal{S}$ whenever $r\in S$. Hence $\mathcal{S}$ is not a network for our topology $\tau$.
  \bigskip

Now, to define the sequences $\{ a_r : r\in \mathbb R\}$ satisfying (*) we first fix a well-ordering $\prec$ of
  $\mathbb R$ in order-type $\mathfrak c$.
 We will also use
 elementary submodels to guide our choices.

    Let $\theta $ be a regular cardinal so that
     $\mathcal P(\mathbb R)\in H(\theta)$.
   Let $\mathfrak M$ denote the set of
   countable elementary submodels $M$ of
    $H(\theta)$ with $(\mathbb R,\prec)\in M$.

     For each
      $M\in \mathfrak M$,
      we define
      $\mathcal S(M) $ to be the family\\
      \centerline{
    $
    \{ T\cap M : T\in M\cap
        [\mathbb R]^{\mathfrak c} \}$.}
In other words, $\mathcal S (M)$ is the countable family consisting of those
countable subsets
of $\mathbb R$ which are the traces on $M$ of the cardinality
$\mathfrak c$ subsets of $\mathbb R$ that happen to be
members of $M$. Note that then $M\cap \mathbb R$ is the largest element of $\mathcal S (M)$, moreover
every such trace  $T \cap M$ is a $\prec$-cofinal subset of $M\cap \mathbb R$.

The map sending each $M\in \mathfrak M$
to $\mathcal S(M)$ is decidedly not 1-to-1 since
$\mathfrak M$ has cardinality at least $2^{\mathfrak c}$,
while
           $ \mathfrak S =   \{ \mathcal S(M) : M\in \mathfrak M\}$ obviously
           has cardinality $\mathfrak c$.

We now formulate and prove a technical lemma  that will play a crucial role in establishing (*).

\begin{lemma}\label{lm:T(M,r)}
Let us fix $M\in \mathfrak M$ and $r\in \mathbb R$, moreover let $\mathcal{T} = \mathcal{T}(M, r)$ be
the family of those $T \in \mathcal S (M)$ for which $r$ is a limit point of every  $\prec$-final segment of $T \cap (q, s)$
for any pair of rationals with $q<r<s$.
Then there is a sequence $a = \{x_n : n < \omega\} \in [M\cap \mathbb R]^\omega$ of $\prec$-order-type $\omega$
that converges to $r$ and satisfies
$a \cap T \ne \emptyset$ for every $T \in \mathcal{T}$.
\end{lemma}

\begin{proof}
First we note that $\mathcal{T}$ is (countably) infinite. Indeed,
for any pair of rationals with $q<r<s$
  the interval $(q,s)$ is an element
  of $M$. It follows that the countable
  set
   $T = (q, s)\cap M$ is an element of
    $\mathcal S (M)$ that is $\prec$-cofinal
    in $\mathbb R\cap M$ together with all its $\prec$-final segments.
    Furthermore,
    for each $x\in
      \mathbb R\cap M$,
     $r$ is a limit point of
     the final segment $\{y \in T : x \prec y\}$ of $T$.
     This clearly yields us infinitely many members of $\mathcal{T}$.

    Let $\{ T_n : n\in \omega\}$ enumerate
all the elements of $\mathcal T$.  Note that then, for every $n \in \omega$,
$\,r$ is a limit point of the $\prec$-cofinal set $(r - 1/n, r + 1/n) \cap T_n$ as well.

It is now straightforward to
choose, by an $\omega$-length recursion,
the required set $a = \{x_n : n < \omega\} \in [\mathbb R\cap M]^\omega$.
Indeed, if $x_{n-1}$ is given then we may just choose $x_n$ from the non-empty set
\ \ $\{y \in (r - 1/n, r + 1/n) \cap T_n : x_{n-1} \prec y\} \setminus \{r\}$.
\end{proof}

We should emphasize that the above defined family $\mathcal{T}(M, r)$ and, hence
the set $a$,
that we chose only depends on $S (M)$,
even though its definition used $M$.

We may now fix a listing $\{ \mathcal S_r : r\in \mathbb R\}$
of $\mathfrak S$  that is $\mathfrak{c}$-abundant, i.e.
for every $\mathcal S \in \mathfrak S$, the set
$\{ r\in \mathbb R : \mathcal S_r = \mathcal S\}$
has cardinality $\mathfrak c$.

\medskip

Given any $r\in \mathbb R$,
we fix  $M\in \mathfrak M$ such that
$\mathcal S(M) = \mathcal S_r$, and then
apply Lemma \ref{lm:T(M,r)} to $M$ and $r$
to find $a_r \in [M\cap \mathbb R]^\omega$ with the properties of $a$ as described there.

Next we show that with this choice $\{a_r : r \in \mathbb{R}\}$
the topology $\tau$ satisfies (*), hence $\tau$ has uncountable net weight.

To see this,
consider any family $\{ T_n :n \in\omega\} \subset
[\mathbb R]^\mathfrak{c}$ and choose any $M\in \mathfrak M$
            so that $\{ T_n : n\in \omega\}\in M$.
            Recall that $T_n^\circ$ denotes the set of all
            complete accumulation points of $T_n$, moreover we have $|T_n \cap T_n^\circ| = \mathfrak{c}$,
            in fact even $|T_n \setminus T_n^\circ| < \mathfrak{c}$. By elementarity we then clearly have
            $T_n \cap T_n^\circ \in \mathcal{T}(M, r)$, consequently $a_r  \cap T_n \ne \emptyset$
            for all $n \in \omega$. It is important to recall here that the family $\mathcal{T}(M, r)$ and hence
            the set $a_r$ is completely determined by $\mathcal S_r = \mathcal S (M)$.

             To complete the proof that (*) holds, we only
have to observe  that  if $\mathcal S \in \mathfrak{S}$ then $X=\{ r  : \mathcal S_r = \mathcal S\}$
             has cardinality $\mathfrak c$.
             \bigskip

             Finally, we need to prove that $(\mathbb R,\tau)$ is in
              $C(\omega_1)$. To do this
              we consider any $\tau$-neighborhood assignment $U$ so that
               $\dom(U)$ is an uncountable subset of
        $\mathbb R$.
        By passing to an uncountable subset we may assume
        that, there is a fixed pair $q<s\in\mathbb Q$
        such that $r\in U(r)\supset (q,s)\setminus a_r$
        for all $r\in \dom(U)$.
     We must then find an uncountable $Y\subset\dom(U)$
    satisfying that $ Y\subset (q,s)\setminus a_{r}$
    for all $r\in Y$. Equivalently, it suffices
    to show that there is an uncountable $Y\subset
    \dom(U)$ satisfying that $y\notin a_r$
    for all distinct $r,y\in Y$. Fortunately, as the $\prec$-order type of each $a_r$ is $\omega$, this is
    an immediate consequence of the following instance
   (Proposition \ref{Erdos-Specker}) of
    the classical
     Erd\H os-Specker theorem in \cite{ErdosSpecker}.
\end{proof}

\begin{proposition}\label{Erdos-Specker}
If $f$ is any function from
     a well-ordered set $(U,\prec)$ of order-type $\omega_1$
     such that $f(x)\subset U$  has $\prec$-order type
     at most $\omega$ for all $x\in U$,
     then there is an uncountable free set
     $Y\subset U$, meaning that $y\notin f(x)$
     for all $x\neq y\in Y$.
     \end{proposition}

\bigskip

Theorem 1.6  of \cite{JSSz} is a strengthening of the above mentioned Theorem 1.5 of \cite{JSSz}. It
says that after adding any number of Cohen reals to a model of CH, in the extension every regular
space in $C(\omega_1)$ belongs to
a class called there
$K(\omega_1)$, for which we have $$N(\omega_1) \subset K(\omega_1) \subset C(\omega_1).$$
We do not give here the, somewhat involved, definition of $K(\omega_1)$  because we shall not need it.
We only mention that, by \cite{JSSz}, it is consistent to have
regular spaces that
belong to  $K(\omega_1) \setminus N(\omega_1)$ or to $C(\omega_1) \setminus K(\omega_1)$.

Problem (2) on p. 4 of \cite{JSSz} asks if Theorem 1.6  of \cite{JSSz} can actually be strengthened to
saying that after adding any number of Cohen reals to a model of CH, in the extension every regular
space in $C(\omega_1)$ belongs even to $N(\omega_1)$. Our next result
yields an affirmative answer to this problem, in fact it yields much more.

\begin{theorem}
If CH holds and we add $\kappa$ Cohen reals, i.e. we force
with $\mathop{Fn}(\kappa,2)$,
then in the extension every regular space in $C(\omega_1)$
has countable net weight.
\end{theorem}

\begin{proof}
Let $X$ be the base set for a
topology and let $\tau$
denote a family
of $\mathop{Fn}(\kappa,2)$-names
for subsets of $X$ that is
forced to be a basis for a regular topology on $X$
that is in the class
  $C(\omega_1)$.
  We may assume that the elements
  $\dot W\in \tau$ are nice names
  in the sense that
  they are subsets
  of $X\times \mathop{Fn}(\kappa,2)$
  and, for
  every $x\in X$, the set
  $\{ p : (x,p)\in \dot W\}$
  is a (possibly empty) antichain.
We may further assume that $\tau$ contains
the set of
all nice names $\dot W$ that have the property
that $1\Vdash \dot W\in\tau$.
Naturally the evaluation, $\mbox{val}_G(\dot W)$,
of such a name $\dot W$ by a generic filter $G$ is
the set $\{ x : ~~(\exists p\in G)~(x,p)\in \dot W\}$.

  \medskip

  Next we choose an elementary
  submodel $M$ of $H(\theta)$,
  for any regular $\theta>2^\kappa$,
  so that $\{X, \kappa,\tau\}\in M$.
  Choose $M$   so  that also $M^\omega\subset M$
  and  $|M|= \mathfrak{c} = \aleph_1$.
  \medskip

  Let $\tau_M$ be the elements
  (names) of $\tau$ that are members
  of $M$ (simply $\tau_M = \tau\cap M$).
  Let $G_M$ be a generic filter
  for $\mathop{Fn}(\kappa\cap M,2)
   = M\cap \mathop{Fn}(\kappa,2)$.
   For each $\dot W\in \tau_M$,
      the set $  \dot W\cap M$
      is a $\mathop{Fn}(\kappa\cap M,2)$-name
      of a subset of $X\cap M$.
   In this extension $V[G_M]$, we
   let $$\sigma =
    \{ \mathop{{val}_{G_M}}(\dot W\cap M)
     :
    \dot W \in \tau\cap M\}.$$

We claim that $\sigma$
     generates a regular topology on $X\cap M$
     in the extension $V[G\sb M]$. Let $x\in X\cap M$
     and $ \dot U$ be an element of $\tau \cap M$
     and let $p\in G_M$ force that $x\in \dot U$.
Since $p$ forces the statement that
 there is a $W\in \tau$ with $x\in W\subset
  \overline{W}\subset \dot U$, we can
  apply the forcing maximum principle
  and our assumption on $\tau$ to
  deduce that there is a $\dot W\in \tau$
  satisfying that $p\Vdash   x\in\dot W
   $ and that $p\Vdash \mbox{cl}_{X}(\dot W)\subset\dot U$.
By elementarity, there is such a $\dot W$ in $M$.
For each $y\in X\cap M$, each of the statements
 $p\Vdash y\in \mbox{cl}_X(\dot W)$ and
  $p\Vdash y\in \dot U$ are absolute
  between $M$ and $H(\theta)$.
  This proves that
   the topology generated by $\sigma$ is regular in the model $V[G_M]$.
   A similar, but even  simpler, argument may be used to prove its Hausdorffness.
\medskip

 We now prove
   that, in the extension $V[G_M]$, the space
   $(X\cap M, \sigma)$ has
   a countable network consisting
   of closures of countable sets.
 Since CH holds in this extension, it
 will suffice for this,  by
          \cite{JSSz}*{Theorem 1.5},
  to prove that
          $(X\cap M, \sigma)$
          has the  $C(\omega_1)$ property (in $V[G_M]$).

         \bigskip

   It will be simplest to work in $V$
   and to suppose that $1$ forces
     $\dot U$
     is a $\mathop{Fn}(\kappa\cap M, 2)$-name
      (nice as usual)
      of a partial
      neighborhood assignment
      into $\tau\cap M $
      and $\mathop{dom}(\dot U)$ is an
      uncountable subset of $X\cap M$.
It should be clear that we may assume that
the elements in the range of $\dot U$
are named by members of  $\sigma$.
      \medskip

We use the assumption that
 $1$ forces over $\mathop{Fn}(\kappa,2)$
 that the space $(X, \tau)$ is in $C(\omega_1)$.
 So, there is a nice name, $\dot Y$,
 for an uncountable subset of
$\mathop{dom}(\dot U)$ witnessing
 $C(\omega_1)$.
       By transfinite recursion
       we may then choose for $\alpha < \omega_1 $
      conditions $p_\alpha\in
      \mathop{Fn}(\kappa,2)$, distinct points
       $y_\alpha\in X\cap M $, and names $\dot W_\alpha \in \tau\cap M$,
       satisfying that $p_\alpha$
       forces that $y_\alpha\in \dot Y$
        and $\dot U(y_\alpha)$ is
        assigned to the name
         $\dot W_\alpha\cap M$.
Let us note that
  $p_\alpha\cap
        M\in M$ forces that $y_\alpha\in
        \mathop{dom}(\dot U)$
        and that $\dot U(y_\alpha)=\dot W_\alpha$.
        \medskip

        Apply a standard $\Delta$-system argument
        to find a root $\bar p\in \mathop{Fn}(\kappa,2)$
        and an uncountable set $\Lambda\subset
        \omega_1$ satisfying that
         $\bar p\subset p_\alpha$
         and $p_\alpha\setminus \bar p$
and $p_\beta\setminus \bar p$ have disjoint
domains for all $\alpha\neq \beta\in \Lambda$.
\medskip

Let us note that   $p_\alpha\cup p_\beta$ forces
 that $y_\alpha \in \dot W_\beta$ for all
  $\alpha,\beta\in \Lambda$.
  Also, by elementarity
we then have that, for all $\alpha,\beta\in \Lambda$,
\begin{enumerate}
\item $(p_\alpha\cap M)\cup (p_\beta\cap M)$
forces that $y_\alpha\in \dot W_\beta$, and
\item $(p_\alpha\cap M)$ forces that
   $\dot U(y_\alpha) = \dot W_\alpha \cap M$.
\end{enumerate}
This completes the proof that $(X\cap M, \sigma)$
is forced by $\mbox{Fn}(\kappa\cap M,2)$ to have
 the $C(\omega_1)$ property.

It follows from \cite{JSSz} that $(X\cap M,\sigma)$
is therefore forced by
 $\mbox{Fn}(\kappa\cap M,2)$ to have a countable network.
Since our space $(X\cap M,\sigma)$ is regular
and hereditarily separable, it also has a countable
network consisting of closed separable sets.
   Using that $M^\omega\subset M$,
    there is a list
    $\{\dot S_n : n\in\omega\}\in M$
    of (nice) names for countable subsets of $X \cap M$
    whose closures taken in $(X\cap M,\sigma)$ form
     a network for $(X\cap M, \sigma)$.
     However it now follows,
     by elementarity, that it is forced
     that
      $\{
      \mbox{cl}_{(X,\tau)}(\dot S_n) :
         n\in \omega\}$
         is a network for all of $(X, \tau)$.
    This is because,
     by the Tarski-Vaught criterion (see \cite{K1}),
      $M[G]$ is an elementary
      submodel of $H(\theta)[G]$.
\end{proof}

\bigskip

Recall that a poset has property $K$ if every
uncountable subset of it has an uncountable subset that
is linked.
The interested reader can check that the only properties
of the poset $\mathbb P = \mbox{Fn}(\kappa,2)$ that were used in
the proof are that
both $\mathbb P$  and $\mathbb P/G_M$ have
property K.   These posets
were similarly utilized in \cite{DH2014}, where they were called {\em finally property K}.

 \bigskip

Hart and Kunen constructed in \cite{HKsup} a consistent example of
a first countable 0-dimensional, hence regular space of net weight $\omega_2$
that is in $C(\omega_1)$. However,
their example is also in $N(\omega_1)$.

In  \cite{JSSz}, Theorem 4.11 it was shown to be consistent that there is a 0-dimensional topology
on $\omega_1$ such that
  a subspace of it has countable net weight iff it is non-stationary.
It is easy to see that such a space is in $C(\omega_1) \setminus N(\omega_1)$.

This led us to consider the natural question if an analogous result could be proved
in which net weight is replaced with weight.
Our next result shows that this is impossible for regular spaces.

\begin{theorem}\label{regthm}
If $\tau$ is a regular
topology of uncountable weight on a stationary subset $S$ of $\omega_1$ then
$S$ has a non-stationary subset that has uncountable weight as well.
\end{theorem}

\begin{proof}
If $S$ has a countable subset of uncountable weight we are done because all countable sets are non-stationary.
So we may assume that all countable subspaces have countable weight.

If $\tau$ is not hereditarily separable then $S$ has an uncountable left-separated subspace,
all whose subspaces have uncountable weight. So, we are done again because every uncountable subset
of $S$ has an uncountable non-stationary subset. So, we may assume that $\tau$ is (hereditarily) separable

This, in turn, implies that $\tau$ is first countable. Indeed, let $D$ be a countable $\tau$-dense
subset of $S$. Then $D \cup \{\alpha\}$ is also $\tau$-dense for every point $\alpha \in S$.
So, since $\tau$ is regular, we have $$\chi(\alpha, S) = \chi(\alpha, D \cup \{\alpha\}) \le \omega$$
by our first assumption. We may thus fix for each $\xi \in S$ a countable $\tau$-neighborhood base $\mathcal{B}_\xi$.

Now, let us fix a continuous increasing elementary chain $\langle M_\alpha : \alpha < \omega_1\rangle$
of countable elementary submodels of $H_\theta$ for a large enough regular cardinal
$\theta$
such that $S, \tau$, and the function sending $\xi \in S$ to $\mathcal{B}_\xi$ all belong to $M_0$.
($\theta = (2^{\omega_1})^+$ will suffice.)
Moreover, let us put $\delta_\alpha = M_\alpha \cap \omega_1$ for each $\alpha < \omega_1$.
It follows from our assumptions that $S \cap \delta_0$ is $\tau$-dense in $S$.

As is well-known, then $C = \{\delta_\alpha : \alpha < \omega_1\}$ is closed and unbounded in $\omega_1$,
consequently $T = S \setminus C$ is non-stationary. Note that
$$T = (S \cap \delta_0) \cup \bigcup\{S \cap (\delta_\alpha, \delta_{\alpha+1}) : \alpha < \omega_1\}.$$
We shall complete the proof by showing that the $\tau$-weight of $T$ is uncountable.

To see this, let us put for each $\alpha < \omega_1$
$$\mathcal{A}_\alpha = \bigcup \{\mathcal{B}_\xi : \xi \in S \text{ and } \xi \le \delta_\alpha\},$$
moreover set $\mathcal{A} = \bigcup \{\mathcal{A}_\alpha : \alpha < \omega_1\}.$
Note that we have $\mathcal{B}_\xi \in M_\alpha$ whenever $\xi \in S \cap \delta_\alpha = S \cap M_\alpha$,
consequently, we have $\mathcal{A}_\alpha \in M_{\alpha+1}$ for all $\alpha < \omega_1$.

Since $\mathcal{A}_\alpha$ is countable for each $\alpha < \omega_1$, it follows that $\mathcal{A}_\alpha$ is not a $\tau$-base,
hence there are some $\eta \in S$ and $V \in \mathcal{B}_\eta$ such that if $\eta \in B \in \mathcal{A}_\alpha$
then $B \setminus \overline{V} \ne \emptyset$. Here we use the regularity of $\tau$ again.
Actually, by elementarity, there is such an $\eta \in S$ satisfying $\delta_\alpha < \eta < \delta_{\alpha + 1}$, hence $\eta \in T$.

Since $S \cap \delta_0$ is $\tau$-dense in $S$, we have $S \cap \delta_0 \cap (B \setminus \overline{V}) \ne \emptyset$ as well,
hence $T \cap (B \setminus \overline{V}) \ne \emptyset$ as well because $S \cap \delta_0 \subset T$ is $\tau$-dense.
This means that such an $\eta \in T$ is a witness for the fact that no family $$\{B \cap T : B \in \mathcal{A}_\alpha\}$$
is a $\tau$-base of $T$, while their union $\mathcal{B} = \{B \cap T : B \in \mathcal{A}\}$ for all $\alpha < \omega_1$
clearly is. Since every countable subfamily of $\mathcal{A}$ is included in some $\mathcal{A_\alpha}$,
it follows that the traces of its members on $T$ does not form a $\tau$-base of $T$.
In other words, no countable subfamily of $\mathcal{B}$ is a $\tau$-base of $T$.

But every base of any space $X$ has a subfamily of size $w(X)$ that is a base of $X$,
consequently we conclude that the $\tau$-weight of $T$ must indeed be uncountable.
\end{proof}

Of course, Theorem \ref{regthm} implies that no regular topology on $\omega_1$ can have the property that
a subspace of it has countable weight iff it is non-stationary.
Our last result extends this for all topologies on $\omega_1$ but with
a considerably harder proof.

\begin{theorem}\label{one}
If $\tau$ is any topology on $\omega_1$
such that every final segment  $[\gamma, \omega_1)$ of $\omega_1$
has uncountable weight then some non-stationary set has uncountable weight as well.
\end{theorem}

\begin{proof}
Just as in the proof of Theorem \ref{regthm},
we may then assume that all countable subsets have countable weight, moreover
that $\tau$ is hereditarily separable, because
any uncountable left separated
subspace  has an uncountable non-stationary subset, which even has uncountable net weight.

\begin{definition} \label{df:bad}
We say that $\alpha\in\omega_1$ is a {\em bad point}
if for every countable collection $\mathcal N\subset \tau$ of neighborhoods of $\alpha$
there is a
non-stationary subset $\alpha\in  A\subset  \omega_1 $ such that
   $\mathcal N\upharpoonright A = \{ U\cap   A
   : U\in \mathcal N\}$ is not a local
   basis at $\alpha$ in the subspace
    $A $.
\end{definition}

Clearly, any bad point must have uncountable $\tau$-character.
We shall use $NS(\omega_1)$ to denote the family of all non-stationary subsets of $\omega_1$.

\begin{lemma}\label{badp}
Assume that no no bad points exist. Then then there is a non-stationary set of uncountable weight.
\end{lemma}

\begin{proof}
By our assumption we may assign to every $\alpha\in \omega_1$
a countable collection of neighborhoods $\mathcal N_\alpha$ such that
$\mathcal N_\alpha \upharpoonright A$ is a local base at $\alpha$ in $A$
whenever $\alpha \in  A \in NS(\omega_1)$.

For any countable subset $\mathcal U \subset \tau$ and for any $\gamma < \omega_1$ we
know that $\mathcal U \upharpoonright [\gamma, \omega_1)$  is not a basis, hence
there is a countable subset $I \subset  [\gamma, \omega_1)$ of limit order type that witnesses this
in the following way: There are $\alpha \in I$ and a neighborhood  $V$ of $\alpha$
so that $I \cap V \setminus U \ne \emptyset$ for all $U \in \mathcal U $.
We denote by $\mathcal W(\mathcal U, \gamma ) $ the set of all such witnesses.

We will produce a non-stationary
set $A$
 that will not
 have a countable basis, contradicting our original assumption.

To achieve this, we define by recursion on $\xi < \omega_1$ countable sets
$I_\xi$ and ordinals $\gamma_\xi$ such that the sequence $\langle\gamma_\xi :  \xi < \omega_1\rangle$
is strictly increasing and continuous, and $\bigcup \{I_\xi : \xi < \eta\} \subset \gamma_\eta$
for all $\eta < \omega_1$.

To begin with, we choose $I_0 \in [\omega_1]^\omega$ of limit order type arbitrarily,
and then set $\gamma_0 = \sup I_0$l.
If $\{I_\xi : \xi < \eta\}$ and $\{\gamma_\xi : \xi < \eta\}$ have been so defined,
then we let $A_\eta = \bigcup \{I_\xi : \xi < \eta\}$ and $\gamma_\eta = \sup A_\eta$, moreover
$\mathcal U _\eta = \bigcup \{\mathcal N _\alpha : \alpha \in A_\eta\}$.
Then we choose $I_{\eta} \in \mathcal W(\mathcal U_\eta, \gamma_\eta ) $.

Clearly, this recursion goes through for all $\eta < \omega_1$, moreover
$A = \bigcup \{I_\xi : \xi < \omega_1\}$ is non-stationary, for it is disjoint
from the cub set $\{\gamma_\xi : \xi < \omega_1\}$. But then for
$\mathcal U  = \bigcup \{\mathcal N _\alpha : \alpha \in A\}$
we have that $\mathcal U \upharpoonright A$ is a basis of $A$, while by our construction
no countable subset of it is, hence $A$ has uncountable weight.
\end{proof}

In view of this we may assume that bad points exist, hence the following
lemma makes sense.

\begin{lemma}\label{getseq}
If $\alpha$ is a bad point
and $\mathcal U \subset\tau$ is any countable collection
of neighborhoods of $\alpha$
then for every $\alpha < \gamma < \omega_1$ there is a countable subset $I$ of the
final segment  $[\gamma, \omega_1)$ such that $\alpha \notin \overline{I}$
but $U \cap I \ne \emptyset$ for all $U \in \mathcal U$.
\end{lemma}

\begin{proof}
Since we know that $ \gamma = [0,\gamma)$ has countable weight,
by enlarging $\mathcal U$ if necessary, we may assume
that $\mathcal U\upharpoonright \gamma$
is a local basis at $\alpha$ in the countable subspace $\gamma$.
By definition, as $\alpha$ is a bad point, there is a
non-stationary set $A$ with $\alpha\in  A$ such that
$\mathcal U\upharpoonright A$ is not a local
basis at $\alpha$ in the subspace $A $.

If $A$ has uncountable weight, the proof of
Theorem \ref{one} is finished. So we have $\chi(\alpha, A) = w(A) = \omega$, hence
we may choose $\mathcal V= \{V_n : n < \omega\} \subset \tau$
such that $\mathcal V\upharpoonright A= \{A \cap V_n : n < \omega\}$ is a decreasing
local base at $\alpha$ in $A$.

On the other hand, $\mathcal U\upharpoonright A$ is not such a local
base, hence there is an open neighborhood $W$ of $\alpha$ for which
$A \cap V_n \setminus W \ne \emptyset$ for all $n < \omega$.
Note that for some $n < \omega$ we have $\gamma \cap V_n \subset \gamma \cap W$,
so we actually have
$(A \setminus \gamma) \cap V_n \setminus W \ne \emptyset$ for all $n < \omega$.

Let us now choose $\xi_n \in (A \setminus \gamma) \cap V_n \setminus W$ for all $n < \omega$
and set $I = \{\xi_n : n < \omega\} \subset \omega_1 \setminus \gamma$. Then $W \cap I = \emptyset$ implies
$\alpha \notin \overline{I}$, and clearly, if $U \in \mathcal U$ then $\xi_n \in U \cap I$ for all but finitely many $n < \omega$.
\end{proof}

In the remaining part of our proof we shall make heavy use of the cub subsets of $\omega_1$.
We denote by $\mathcal C$ the family of all cub subsets of $\omega_1$,
moreover we make use of
the following standard fact.

\begin{lemma} Let $f$ be any function from \label{diagonal}
$\mathcal C$ into $\omega_1$. Then there
is a $\gamma\in\omega_1$ such that for any  $D \in \mathcal C$ there is
$C \in \mathcal C$ with  $f(C) = \gamma$ and $|C \setminus D| \le \omega$.
\end{lemma}

\begin{proof} Suppose that, for each $\delta < \omega_1$, there
is a cub $C_\delta$ satisfying that no member
of $f^{-1}( \delta)$ is contained, mod countable,
in $C_\delta$.  Let $\bar C$ denote the diagonal
intersection $\{ \gamma\in\omega_1 :
   \delta < \gamma \Rightarrow \gamma\in C_\delta\}$. Then $\bar C \in \mathcal C$,
so we let $f( \bar C) = \bar \delta$.
But $\bar C\setminus C_{\bar \delta}$ is countable,
contradicting our choice of $C_{\bar \delta}$.
\end{proof}

Next we choose for each $C \in \mathcal C$ a countable elementary submodel
$M_C$ of a suitable $H_\vartheta$ such that $C, \tau \in M_C$ and
then set $\delta_C = M \cap \omega_1$.

Also, we recall that our topology $\tau$ is hereditarily separable.
Thus for every open set $U \in \tau$ we may fix a countable subset $H(U)$ of its complement
$\omega_1 \setminus U$ that is dense in it. Note that $H(U)$
determines $U$ because $U = \omega_1 \setminus  \overline{H(U)}$.

We claim that there is $C \in \mathcal C$ such that
the non-stationary set $A_C = \delta_C\cup (\omega_1\setminus C)$
has uncountable weight, thus completing our proof.

Assume, on the contrary, that for every $C \in \mathcal C$
we may choose a countable family $\mathcal U_C\subset \tau$
so that it traces a countable base for the subspace topology
on $ A_C$. We shall show that this leads to a contradiction.

To arrive at a contradiction, we now define a function $f$ from $\mathcal C$ into
$\omega_1$.
Given $C \in \mathcal C$ we can choose $f(C) < \omega_1$ so that
$\delta_C \le f(C)$ and $H(U) \subset f(C)$ for all $U \in \mathcal U_C$.

\medskip

Applying Lemma \ref{diagonal} then we may
fix $\zeta\in\omega_1$ so that
for any  $D \in \mathcal C$ there is
$C \in \mathcal C$ with  $f(C) = \zeta$ and $|C \setminus D| \le \omega$.
We then let $\delta = \min \{\delta_C : C \in C_\zeta = f^{-1}(\zeta)\}$
 and choose
$C_1\in\mathcal C_\zeta$ so that $M_{C_1}\cap \omega_1=\delta$.

Notice that
since $\tau\in M_{C_1}$, by elementarity, the non-empty
set of bad points is an
element of $M_{C_1}$.
Therefore we may fix a bad point $\alpha < \delta$.
We also note that $  \delta \leq
\zeta$.

We can also fix a countable family $\mathcal B\subset\tau$
that traces a countable base for the subspace topology
on $\zeta$.

Since $\alpha < \delta \le \zeta$, we may then apply Lemma \ref{getseq}
 to choose a countable set $I_0\subset \omega_1\setminus \zeta$ so
 that $\alpha \notin \overline{I_0}$ but
$B \cap I_0 \ne \emptyset$ whenever
 that $\alpha\in B \in \mathcal B$. We continue with an
 $\omega_1$-length recursion as follows.
If we have
 chosen $\{I_\xi : \xi < \eta\}$, then we set
  $\gamma_\eta = \min \{\gamma : \bigcup\{I_\xi : \xi <\eta\} \subset \gamma\}$.
We then choose, by Lemma \ref{getseq},
  a countable set $I_\eta \subset (\gamma_\eta, \omega_1) $
  so that again $\alpha \notin \overline{I_\eta}$ while $B \cap I_\eta \ne \emptyset$ whenever
$\alpha\in B \in \mathcal B$.

 Let  $A=\bigcup\{ I_\xi : \xi<\omega_1\}$
 and note that $A$ is disjoint from the cub set
  $\Gamma = \{ \gamma_\eta : \eta\in\omega_1\}$.
 Therefore, there is a
  cub set $C\in \mathcal C_\zeta$
  such that $C\cap A$ is countable. So, we may
  choose $\eta$ large enough so that $I_\eta \subset \omega_1 \setminus C$.
  Recall that  $\mathcal U_C\subset \tau$
traces a base for the subspace topology  on
  $A_C\supset \delta \cup (\omega_1\setminus C)$.
  We can therefore choose  $U \in \mathcal U_C$ so
  that $\alpha\in U$ and $U \cap I_\eta = \emptyset$ because $\alpha \notin \overline{I_\eta}$.
 It follows that $I_\eta$ is a subset of the closure of
 $H(U)$, while $f(C) = \zeta$ implies that $H(U) \subset \zeta$.

 Now we may choose $B \in \mathcal B$ so that $\alpha\in B$
 and $B\cap \zeta \subset U \cap \zeta$, and so we have
 $B \cap H(U) = \emptyset$. But this clearly implies $B \cap \overline{H(U)} = \emptyset$
 as well, hence $B \subset U$. This, however, is impossible because, by our construction,
 $B \cap I_\eta \ne \emptyset$ while $U \cap I_\eta = \emptyset$,
 providing us with the desired contradiction.

\end{proof}

{\bf Acknowledgement} The research on this paper was started when the second author visited
the Department of Mathematics and Statistics of UNC Charlotte. He would like to thank this
institution as well as the NKFIH grant no. 149211 for financially supporting this visit.


\begin{bibdiv}

\begin{biblist}



\bib{DH2014}{article}{
   author={Dow, Alan},
   author={Hart, Klaas Pieter},
   title={Reflecting Lindel\"of and converging $\omega_1$-sequences},
   journal={Fund. Math.},
   volume={224},
   date={2014},
   number={3},
   pages={205--218},
   issn={0016-2736},
   review={\MR{3194414}},
   doi={10.4064/fm224-3-1},
}

\bib{ErdosSpecker}{article}{
   author={Erd\H os, P.},
   author={Specker, E.},
   title={On a theorem in the theory of relations and a solution of a
   problem of Knaster},
   journal={Colloq. Math.},
   volume={8},
   date={1961},
   pages={19--21},
   issn={0010-1354},
   review={\MR{0130182}},
   doi={10.4064/cm-8-1-19-21},
}



\bib{HKsup}{article}{
   author={Hart, Joan E.},
   author={Kunen, Kenneth},
   title={Super properties and net weight},
   journal={Topology Appl.},
   volume={274},
   date={2020},
   pages={107144, 19},
   issn={0166-8641},
   review={\MR{4073574}},
   doi={10.1016/j.topol.2020.107144},
}

 \bib{HKsupdup}{article}{
   author={Hart, Joan E.},
   author={Kunen, Kenneth},
   title={Superduper properties},
   journal={Topology Proc.},
   volume={60},
   date={2022},
   pages={261--274},
   issn={0146-4124},
   review={\MR{4478135}},
}

\bib{K1}{book}{
   author={Kunen, Kenneth},
   title={Set theory},
   series={Studies in Logic and the Foundations of Mathematics},
   volume={102},
   note={An introduction to independence proofs},
   publisher={North-Holland Publishing Co., Amsterdam-New York},
   date={1980},
   pages={xvi+313},
   isbn={0-444-85401-0},
   review={\MR{0597342}},
}

\bib{JSSz}{article}{
author={I.~Juh{\'a}sz},
author={L.~Soukup},
author={Z.~Szentmikl{\'o}ssy},
title = { The class $C(\omega_1)$ and countable net weight},
journal = {Top. Appl.},
note = {Available online },
doi= {10.1016/j.topol.2024.109106}}


\bib{Tkachenko78}{article}{
   author={Tka\v cenko, M. G.},
   title={Chains and cardinals},
   language={Russian},
   journal={Dokl. Akad. Nauk SSSR},
   volume={239},
   date={1978},
   number={3},
   pages={546--549},
   issn={0002-3264},
   review={\MR{0500798}},
}
\end{biblist}

\end{bibdiv}
\end{document}